\documentclass[11pt]{amsart}
\def\supp{\operatorname{supp}}

\def\cl{\operatorname{cl}}

\newtheorem{theorem}{Theorem}
\newtheorem{proposition}{Proposition}

\newtheorem{corollary}{Corollary}
\theoremstyle{remark}

\usepackage{epsfig}
\usepackage{graphicx} 

\begin{document}

\title{With Andrzej Lasota there and back again}
\author{Ryszard Rudnicki}
\address{R. Rudnicki, Institute of Mathematics,
Polish Academy of Sciences, Bankowa 14, 40-007 Katowice, Poland.}
\email{rudnicki@us.edu.pl}

\begin{abstract}
The paper below is a written version of the 17th Andrzej Lasota Lecture presented on January 12th, 2024 in Katowice. 
During the lecture we tried to show the impact of Andrzej Lasota's results on the author's research concerning various fields of mathematics, including chaos and ergodicity of dynamical systems,
Markov operators and semigroups and partial differential equations.
\end{abstract}

\maketitle

\section{Introduction}
\label{s:intro}
Let us start with a brief introduction of Professor Andrzej Lasota 
(1932--2006). He studied physics and mathematics at  Jagiellonian University (1951--1955).  From 1955--1975 he worked at Jagiellonian University, where he was the Dean of the Faculty of MPCh (1972--1975). 
 
He then moved to Katowice at the University of Silesia, where he worked until his death with a short break in the years of 1986--1988 spent at Maria Curie-Sk{\l}odowska University in Lublin.
During that time he had also part-time positions at 
Jagiellonian University (1975--2005) and Institute of Mathematics Polish Academy of Sciences (1995--2006).

His research interests concerned the following areas of mathematics: 
\begin{itemize}
\item[(1)] differential equations,
\item[(2)] dynamical systems: chaos, ergodicity and fractal theory,
\item[(3)] Markov operators,
\item[(4)] applications of mathematics to biology, physics and technology.
\end{itemize}

I met Professor Lasota at his lecture on the theory of differential equations in 1977, and after the exam, I began my research under his supervision. This was a period when he conducted very intensive research dealing with a variety of issues in pure and applied mathematics collaborating with well-known scholars, including J.A. Yorke -- American mathematician co-founder of chaos theory and M.C. Mackey -- Canadian biomathematician 
(see Fig.~\ref{r:zdjecie}).  This gave me the opportunity to learn about top scientific issues and join in this research.   

\begin{figure}[h]
\begin{center}
\includegraphics[height=10cm]{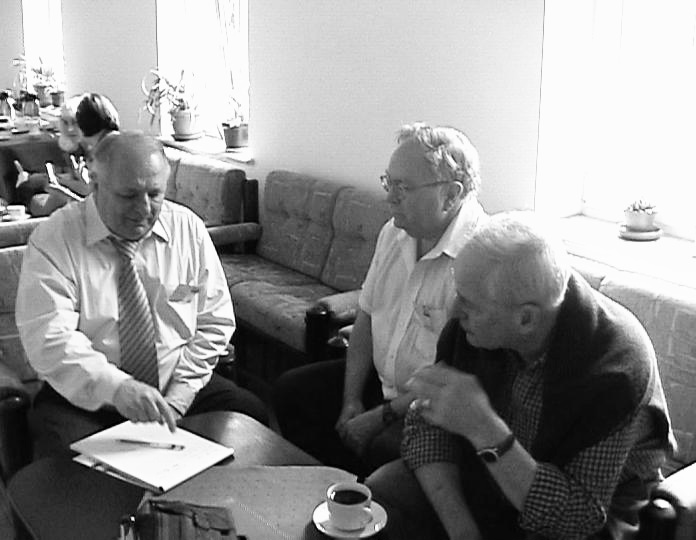}
\end{center}
\vskip1mm
\caption{Special Scientific Session  in honour of Professor Andrzej Lasota on the occasion of his 70th birthday, B\c edlewo 22.06.2002. From left A. Lasota, J.A. Yorke, M.C. Mackey.}
\label{r:zdjecie}
\end{figure}

During the lecture, I tried to show the impact of Andrew Lasota's results on my research. I will focus on a few selected topics, mainly on chaos theory for partial differential equations and asymptotics of Markov operators and semigroups and their applications in biology.

\section{Chaos and ergodicity of dynamical systems}
\label{s:chaos}
\subsection{General remarks}
\label{ss:gen}
Since the issues of chaos and ergodicity will occur at various points in the lecture, we will now collect the necessary definition and basic results. A more extensive introduction to these issues is presented in the survey papers~\cite{R-chaos-p,Rudnicki-er-chaos}.

Let $(X,\rho)$ be a metric space. 
A \textit{semiflow} $\{S^t\}_{t\ge 0}$   on $X$ is a collection of transformation 
$S^t\colon X\to X$, for $t\ge 0$, such that
\begin{itemize}
\item[(a)] $S^0=\operatorname{Id}$,  $\,S^{t+s}=S^t\circ S^s$ for $t,s\ge 0$,
\item[(b)] $S\colon [0,\infty)\times X\to X$, defined as $S(t,x):=S^t(x)$, is a continuous function of $(t,x)$.
\end{itemize}
The set  $\mathcal O(x)=\{S^t(x)\colon t\ge 0\}$ is called an orbit  of the point $x$.

For example if $f\colon \mathbb R^n\to \mathbb R^n$ is a  Lipschitz function 
then the problem 
\[
x'(t)=f(x(t)) ,\quad  x(0)=x_0\in \mathbb R^n  
\]
has unique solution $x \colon [0,\infty)\to\mathbb R^n$
and $S_t(x_0)=x(t)$ defines a semiflow on 
$\mathbb R^n$.  

If $S\colon X\to X$ is a continuous map, then the sequence of its iterates $\{S^n\}_{n=0}^{\infty}$ creates
a \textit{discrete time semiflow.}
Discrete and continuous time semiflows 
belong to a broader class called 
\textit{dynamical systems}.
In a dynamical system, instead of the continuity condition (b), for example, we can assume the measurability of a function
$S(t,x)$.

When  can we say that a dynamical system is chaotic? General answer is that
it has a simple 
and {\bf deterministic} description, but it behavies 
in a  complicated and {\bf random} way.
Random means here that  the semiflow has unpredictable behaviour. 
The semiflow is very sensitive on small  perturbations of initial data
and  looks like a system  with a stochastic noise. 

Almost all specialists are willing to accept this not too precise definition.
But,  in fact,  there are hundreds of  definitions of chaos.
Generally we have three different approaches  to chaos.

1. \textbf{Macroscopic approach:} 
the existence of global attractors with complicated structure (strange attractors).

2. \textbf{Microscopic approach:} 
the existence of trajectories, which are
unstable, turbulent or dense in the phase space;
 topologically mixing.

3. \textbf{Stochastic approach:} 
the existence of invariant measures having
strong ergodic and analytic properties.

It should be noted that our ``classification" of chaotic behaviour  is completely arbitrary, but it could help non-specialists what kind of problems are studied in the theory of chaos. In the lecture, I will focus on the connections between the second and third approaches to justify that the methods of ergodic theory
 allow us to obtain quite strong chaotic properties.

\subsection{Microscopic approach}
\label{ss:micro}
In this approach we have several different definitions of chaos.
A flow is \textit{chaotic  in the sense of  Auslander--Yorke 
\cite{AY}} if 
\begin{itemize}
\item[(a)] there exists a dense trajectory,    
\item[(b)] each trajectory is unstable.
\end{itemize}

Let us mention that a  semiflow on some topological vector space $X$
having  a dense orbit is said to be \textit{hypercyclic} \cite{HW}, which in the case when $X$ is  a Baire space it is equivalent to
the flow being  topologically transitive.
We consider a strong version of unstability, which is
also called \textit{sensitive dependence on initial conditions}: 
there exists a constant $\eta>0$ such that for each 
point $x\in X$ and for each $\varepsilon>0$ there exist a point 
$y\in B(x,\varepsilon)$ and  $t>0$ such that 
$\rho(S^t(x),S^t(y))>\eta$. 
Here $B(x,r)$ denotes the open ball in $X$ with centre $x$ and radius $r>0$.

A little stronger definition of  chaos was given by
Devaney \cite{Devaney}:
\begin{itemize}
\item[(a')] there exists a dense trajectory,  
\item[(b')] the set of periodic points is dense in $X$.
\end{itemize}
One can check that from (a') and (b') it follows that
each trajectory is unstable.

A semiflow  $\{S^t\}_{t\ge 0}$ is
\textit{topologically mixing} if for
any two non-empty open sets $U$, $V$ of $X$ there exists
$t_0>0$ such that $S^t(U)\cap V\ne\emptyset$ for  $t\ge  t_0$.
The topological mixing  is a strong chaotic property of a transformation
and implies chaos in the sense of  Auslander--Yorke and 
sensitive dependence on initial conditions.

One of the chaotic properties of a semiflow is turbulence.
A trajectory $\mathcal O(x)=\{S^t(x):t\ge 0\}$ of the point $x$
 is \textit{turbulent in the sense of Lasota--Yorke}
\cite{LY},
if its closure 
$\cl (\mathcal O(x))$
 is a compact set  and 
$\cl (\mathcal O(x))$
does not contain periodic points.

\begin{theorem}
\label{turbulentnaLY}
Let $\{S^t\}_{t\ge 0}$ be a semiflow.
If for some nonempty, compact, and disjoint sets $A$ and $B$
we have
\begin{equation}
\label{c:LY}
A\cup B\subset S^{t_0}(A)\cap S^{t_0}(B)\quad\textrm{for  some $t_0>0$,}
\end{equation}
\label{twergodyczne2}
then  there exists a turbulent orbit in the sense of
Lasota--Yorke.
\end{theorem}
Theorem \ref{turbulentnaLY} was proved in
\cite{LY} (a discrete time version) and in \cite{Lasota-turbulencja}
(a continuous time version).

\subsection{Stochastic approach} 
\label{ss:stoch}
We recall some necessary definitions.

By a measure on $X$ we mean any probability measure
defined on the $\sigma$-algebra $\mathcal{B}(X)$ of Borel
subsets of $X$. By $\supp\mu$ we denote the topological support of
the measure $\mu$.

A measure $\mu$ is called \textit{invariant} under
a semiflow $\{S^t\}_{t\ge 0}$, if
$\mu(A)=\mu(S^{-t}(A))$ for each $t\ge 0$ and each
Borel subset~$A$.  Here  $S^{-t}(A):=(S^t)^{-1}(A)$.
We will denote a semiflow
$\{S^t\}_{t\ge 0}$ with an invariant measure $\mu$ by
$(S,\mu)$. 

The semiflow $(S,\mu)$
is called \textit{ergodic} if
for each $\mu$-integrable function $f\colon X\to \mathbb R$
we have
\begin{equation}
\label{twergodyczne}
\lim_{T\to\infty}\frac1T\int_0^T f(S^t(x))\,dt=\int_Xf(x)\,\mu(dx)
\quad \text{for  $\mu$-a.e. $x$}.
\end{equation}
If we substitute $f=\mathbf{1}_A$ in equation (\ref{twergodyczne}),
then the left-hand side of (\ref{twergodyczne})
is the mean time of visiting
the set $A$ and
the right-hand side of (\ref{twergodyczne})
is  $\mu(A)$.

A semiflow $(S,\mu)$
is called \textit{mixing} if
\begin{equation}
\label{mixing}
\lim_{t\to\infty} \mu (S^{-t} (A) \cap B)=\mu (A)\mu (B)
\end{equation}
for all $A,B\in\mathcal \mathcal{B}(X)$.
If ${\rm P}=\mu$  and ${\rm P}(B)>0$ then 
condition (\ref{mixing}) can be written in the following way
\[   
\lim_{t\to\infty} {\rm P}(S^t(x)\in A|x\in B)=\mu(A)\quad
\text{for  all $A \in \mathcal{B}(X)$},
\]
which means that 
the trajectory of almost all points enters 
a set $A$ with  asymptotic probability $\mu(A)$.

The  semiflow $(S,\mu)$
is \textit{exact}
if for each measurable set $A$ with
$\mu(A)>0$ we have
\begin{equation}
\label{exactness}
\lim_{t\to\infty}\mu(S^t(A))=1.
\end{equation}
We have: exactness $\Rightarrow $ mixing $\Rightarrow $ ergodicity.

We now show how ergodic properties of a  semiflow
imply its chaotic behaviour. 
\begin{proposition}
\label{mixingchaos}
Let $X$ be a separable metric space and $\mu$ be a probability Borel measure on $X$ such that $\,\supp\mu=X$.
If a  semiflow $(S,\mu)$ is ergodic, then  the orbit of $\mu$--a.e. point
$x$ is  dense in $X$.
If a  semiflow $(S,\mu)$ is mixing,
then the semiflow
$\{S^t\}_{t\ge 0}$
is topologically mixing, in particular $(S,\mu)$ has the property of sensitive dependence on initial conditions.
\end{proposition}

Sometimes the combination of the methods of ergodic theory and dynamical systems allows us to obtain stronger properties of the considered systems. We combine the following well known
Krylov--Bogoliubov theorem with condition \eqref{c:LY}.
\begin{theorem}
\label{miaraKB}  
Let $\{S^t\}_{t\ge 0}$ be a semiflow
on a compact metric space. Then there exists a probability Borel 
measure 
$\mu$  invariant and  ergodic with respect to~$\{S^t\}_{t\ge 0}$.
\end{theorem}

This theorem is a nice general result, but if the semiflow has 
a periodic point $x$, this result is trivial, because we find an invariant ergodic measure supported on the orbit of $x$.      
To formulate a stronger version of Theorem~\ref{miaraKB} we need 
the following definition. 
A measure $\mu$ is said to be \textit{continuous}
if $\mu(\mathrm{Per})=0$, where
Per is the set of all periodic points of the semiflow $\{S^t\}_{t\ge0}$.
\begin{theorem}
\label{miaraKB-w}
Let $\{S^t\}_{t\ge 0}$ be a semiflow
on a compact metric space.  If there are
nonempty, compact, and disjoint sets $A$ and $B$
satisfying condition  $(\ref{c:LY})$,
then there exists a continuous probability Borel
measure
$\mu$  invariant and  ergodic with respect to~$\{S^t\}_{t\ge 0}$.
\end{theorem}

\begin{proof}
From Theorem \ref{turbulentnaLY} it follows the existence of an
 orbit $\mathcal O(x)$  turbulent
in the sense of Lasota--Yorke. We now consider the
 semiflow restricted to the orbit $\mathcal O(x)$.
By  Theorem~\ref{miaraKB}
there exists a probability invariant ergodic measure $\mu$ supported on
$\mathcal O(x)$. Since the orbit $\mathcal O(x)$
has no periodic points the measure $\mu$ is continuous.
\end{proof}
 
\section{Frobenius--Perron operator and ergodic properties}
\label{s:F-P-op}
\subsection{Introduction}
\label{ss:F-P-intro}
Ergodic properties of semiflows can be successfully investigated using methods of operator theory~\cite{LM}. 
We now introduce the notion of Frobenius--Perron operator
and show how ergodic properties of flows
can be expressed in the language of 
Frobenius--Perron operators.

Let $(X,\Sigma,m)$ be a $\sigma$-finite measure space.
A measurable map $S\colon X\to X$  
is called \textit{non-singular}
if it satisfies the following condition 
\begin{equation}
\label{trana2}
m(A)=0 \Longrightarrow  m(S^{-1}(A))=0\textrm{ for $A\in \Sigma$.}
\end{equation}
Let $L^1=L^1(X,\Sigma,m)$ and let $S$ be a nonsingular 
map of $X$.
An operator  $P_S\colon  L^1\to L^1$
which satisfies the following condition
\begin{equation}
\label{def-FP}
\int_A P_S f(x)\,m(dx)=
\int_{S^{-1}(A)} f(x)\,m(dx)
\textrm{\,  for $A\in \Sigma$ and $f\in L^1$}
\end{equation}
is called the
{\it Frobenius--Perron operator} for the transformation~$S$. 
The operator $P_S$ is linear, \textit{positive}\index{positive!operator} 
(if $f\ge 0$ then $P_Sf \ge 0$) and preserves the integral    
($\int_XP_Sf\,dm=\int_Xf\,dm $).
The adjoint of the Frobenius--Perron operator
$P_S^*\colon L^{\infty}\to L^{\infty}$ is 
given by $P_S^*g(x)= g(S(x))$.

Denote by  $D$ the set of all densities with respect to $m$,
i.e. functions $f\in L^1(X,\Sigma,m)$ such that $f\ge 0$ and $\|f\|=1$.
Let    $S\colon X\to X$ be a nonsingular transformation 
and let $f^*\in D$. Then the measure 
$\mu(A)=\int_Af^*\,dm$ for $A\in \Sigma$,
 is invariant with respect to $S$  
if and only if $P_Sf^*=f^*$.

Now we consider a family of transformations  $\{S^t\}_{t\in \mathcal T}$, where $\mathcal T=[0,\infty)$ or $\mathcal T=\mathbb N$.
If the map
$S\colon \mathcal T\times X\to X$ is measurable 
 and each transformation $S^t$ is nonsingular, then  we denote by $P^t$  the Frobenius--Perron operator corresponding to $S^t$.
Then  the quadruple $(X,\Sigma,\mu,S^t)$ 
is a measure-preserving dynamical system
if and only if $P^t\!f^*=f^*$ for all $t\in \mathcal T$.
We collect the relations between ergodic properties of dynamical systems 
$(X,\Sigma,\mu,S^t)$ and the behavior of the 
Frobenius--Perron operators $\{P^t\}$ in Table~\ref{tab1}.

\begin{table}[h]
\begin{center}
\begin{tabular}{|c|c|}
\hline
$\mu$ & $f^*$
\\
\hline
invariant&$P^t\!f^*=f^*$ for all $t\in \mathcal T$\\
ergodic&$f^*$ is a unique fixed point  in $D$ of all $P^t$ \\
mixing& w-$\lim_{t\to\infty} P^t\!f=f^*$ for every $f\in D$\\
exact & $\lim_{t\to\infty} P^t\!f=f^*$ for every $f\in D$\\
\hline
\end{tabular}
\end{center}
\vskip2mm
\caption{${}$}
\label{tab1}
\end{table} 

We recall that the \textit{weak limit}\index{weak limit}  w-$\lim_{t\to\infty} P^t\!f$ is a function $h\in L^1$  such that for 
every $g\in L^{\infty}$ we have 
\[
\lim_{t\to\infty}\int_X  P^t\!f(x)g(x)\,m(dx)=\int_X  h(x)g(x)\,m(dx).
\]

\subsection{Lasota--Yorke lower function theorem}
\label{ss:lower}
A. Lasota and J.A. Yorke have developed tools to study the convergence of Frobenius--Perron operators and there applications to expanding maps~\cite{LY-TAMS1}. 
One of such tools is the 
\textit{lower function theorem}~\cite{LY-TAMS2}. 
This theorem can be applied to a broader class of operators called Markov or stochastic operators, so we will start with a definition of them.

Let $(X,\Sigma,m)$ be a $\sigma$-finite measure space.
A linear operator $P\colon  L^1\to L^1$ is called a
\textit{Markov} (or \textit{stochastic}) \textit{operator} 
 if $P(D)\subset D$. 
 
The family $\{P^t\}_{t\ge0}$ of Markov operators is called a 
\textit{stochastic semigroup}
if it satisfies the following conditions:
\begin{itemize}
\item[(a)] \  $P(0)=I$,  i.e., $P(0)f =f$,
\item[(b)] \ $P^{t+s}=P^t P^s\quad \textrm{for}\quad
s,\,t\ge 0$,
\item[(c)] \ for each $f\in L^1$ the function
$t\mapsto P^tf$ is continuous.
\end{itemize}
We also consider a discrete time stochastic semigroup defined as the iterates of a Markov operator.

Markov operators  appear in ergodic theory of dynamical systems and iterated function systems. They also describe the evolution of Markov chains.
Stochastic semigroups  describe the evolution of densities of  distributions of Markov processes like diffusion processes, piecewise deterministic processes and hybrid stochastic processes.

Consider a stochastic semigroup 
$\{P^t\}_{t\ge0}$.
A density $f^*$
is called \textit{invariant\,}\index{invariant density}
 if $P^tf^*=f^*$ for each $t>0$.
The stochastic semigroup
$\{P^t\}_{t\ge 0}$
is called {\it asymptotically stable\,}\index{asymptotically stable stochastic semigroup}
 if there is an invariant density
$f^*$ such that
\[
\lim _{t\to\infty}\|P^tf-f^*\|=0 \quad \text{for}\quad
f\in D.
\]

 A function $h\in L^1$,  $h\ge 0$ and $h\ne 0$, is called 
a \textit{lower function}
for a stochastic semigroup 
$\{P^t\}_{t\ge 0}$  if 
\[
P^tf(x)\ge h(x)-\varepsilon_t(x)\quad\text{and}\quad 
\|\varepsilon_t\|\to 0
\] 
for  each density  $f$.
This condition  can be equivalently written as: 
\[
\lim_{t\to\infty}\|(P^tf-h)^-\|=0\quad\textrm{for every $f\in D$.}
\]
Observe that if the semigroup  is asymptotically stable
then its invariant density $f^*$ is a lower function for it.
Lasota and Yorke  \cite{LY-TAMS2} proved the following converse result.
\begin{theorem}
\label{LYt} 
If there exists  a lower function for a  stochastic semigroup $\{P^t\}_{t\ge 0}$
then this semigroup is asymptotically stable.
\end{theorem}

\subsection{An generalization of the lower function theorem}
\label{ss:g-lower}
G. Pianigiani and J.A. Yorke~\cite{PianigianiYorke}
considered the following problem.
Let $X$ be a compact subset  of
$\mathbb R^n$ and
let
$S\colon X\to \mathbb R^n$ be 
a piecewise expanding map.
We consider the sequence of iterates 
$x,S(x),\dots,S^t(x)$ 
of points from $X$
as long as they still remain in the set $X$. 
If $S^t(x)\notin X$, then we lose the orbit of $x$.
We consider the following problem.
If the initial distribution  
of points is given by a probability measure $\mu_0$,
what is the conditional probability 
$\mu_t$ that $S^t(x)\in A$ 
given that $x,S(x),\dots,S^t(x)$ are in $A$? 
Such a problem can be called a \textit{discrete time billiard with holes},
because we can treat $X$ as a table with holes and 
the transformation $S$ describes the motion of a ball 
which  can fall through some hole.

We can  define the Frobenius--Perron
operator $P_S$  for the map $S$.
The operator $P_S$ is positive,
 but since the set $S(X)$ is not containing in $X$,
the operator $P_S$ does not preserve the integral.  
If $\mu_0$ has a density $f$, then $\mu_t$ has density 
$P_S^tf/\|P_S^tf\|$ and $\|P_S^tf\|$ is the fraction of points
such that the sequence $x,S(x),\dots,S^t(x)$ 
lies in the set $A$.

Since Theorem~\ref{LYt} was useful to study asymptotic stability  
of Frobenius--Perron operators of  piecewise expanding maps 
$S\colon X\to X$, 
we hoped that it is possible to find a version of this theorem 
which can be successfully applied to billiards with holes.   
The problem was solved in~\cite{Rudnicki-C(X)}, 
but instead of operators on the space $L^1$ we consider operators on $C(X)$. The result is the following.
\begin{theorem}
\label{RR-lf}
Let $X$ be a compact Hausdorff space,
$P$ be a positive operator on $C(X)$,   
$F$ be a dense subset of $C_+ (X)=\{f\in C(X)\colon \min f>0$\}, and $\alpha>0$.
We assume that for each $f\in F$ there exists a constant $n_0(f)$
such that 
\[
P^n f / \| P^n f \| \ge \alpha \quad \text{for} \quad
n\ge n_0 (f).
\]
We also assume that for some  $g\in C_+ (X)$, the sequence  
$\{P^n g /\| P^n g \| \}$ 
has a weakly convergent subsequence.
Then there exist a probability measure  $\nu$, $\lambda > 0$, and a function $f^*\in C_+(X)$ such that $Pf^* = \lambda f^*$, $\nu(f^*)=1$ and
\[
\lim_{n\to\infty} \| \lambda^{-n} P^n f-f^* \int_Xf(x)\,\nu (dx) \| = 0 
\text{\ for each\ }  f\in C (X).
\]
\end{theorem}

\begin{figure}[h]
\begin{center}
\includegraphics[height=5cm]{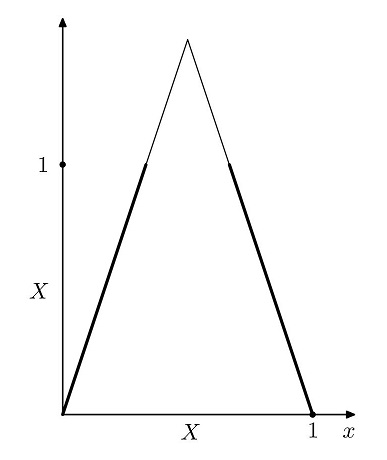}
\end{center}
\vskip-5mm
\caption{}
\label{r:zjazd1}
\end{figure}
Consider the following example:  
$S\colon [0,1]\to [0,\infty)$ is given by 
$S(x)=c \min\{x,1-x\}$, $c>2$
(see Fig.~\ref{r:zjazd1}). Then
\[
P_Sf(x)=\frac 1c\left(f\left(\frac xc \right)+f\left(1-\frac xc\right) \right).
\]
In this example we have 
\[
\lambda^{-n} P^n_S f\to f^* \int_Xf(x)\,\nu (dx),
\]
$\lambda=2/c$ and the support of the measure $\nu$ is similar to the Cantor's set.

A continuous version of Theorem~\ref{RR-lf} was proved
in~\cite{LasotaRudnicki} and applied to study 
the movement of a point on a circle with random jumps.  
Nowadays, such a movement is called a piecewise deterministic Markov process, although at the time we did not know this name.
We also applied this theorem to some partial differential equations with integral perturbation.
Methods presenting in papers 
\cite{LasotaRudnicki,Rudnicki-C(X)}
are also used in the paper with
M.C. Mackey~\cite{MackeyRudnicki94sim}
to study the asymptotic behaviour of non-homogeneous age dependent cellular populations.
A non-autonomous version of Theorem~\ref{RR-lf} was proved by 
A. Lasota and J.A. Yorke~\cite{LY-SIAM}.

\section{Asymptotic behaviour of stochastic semigroups}
\label{s:Asymptot}
\subsection{Partially integral stochastic semigroups}
\label{ss:partial-integral}
We will now present methods for studying the asymptotics of stochastic semigroups based on an idea reminiscent of the use of the lower function.
It involves estimating a semigroup from below using an integral operator. 

A Markov semigroup $\{P^t\}$ is  called  \textit{partially integral}
if there exist $t>0$ and a measurable function   $k(t,x,y)\ge 0$, such that 
\[
\int_X\int_X k(t,x,y)\,m(dx)m(dy)>0,
\]
and
\[
P^tf(y)\ge\int k(t,x,y)f(x)\,m(dx)
\quad {\rm for}
\quad f\in D.
\]

Using the theory of Harris operators one can prove many interesting results 
concerning partially integral 
stochastic (and substochastic) semigroups.

The basic results have been published in the paper~\cite{R-b95}, but we will limit ourselves to a few selected later results.
In \cite{PR-jmaa2} was proved the following result.
\begin{theorem}
\label{asym-th2} 
If  a continuous time partially integral stochastic semigroup $\{P^t\}_{t\ge 0}$
has a unique invariant density $f^*$ 
and  $f^*>0$,
then it 
is asymptotically stable.   
\end{theorem}

\subsection{Asymptotic decomposition of  a stochastic semigroups}
\label{ss:asym-decomp}
The main problem with applying Theorem \ref{asym-th2} 
is that often it is not easy to find an invariant density.
To get around this difficulty, in \cite{PR-JMMA2016} was  devised a method of studying of stochastic semigroups via their asymptotic decomposition.
To present this method, from now on we assume additionally  that $(X,\rho)$ is a separable metric
space and   $\Sigma=\mathcal B(X)$ is the $\sigma$-algebra of Borel subsets of $X$. We consider stochastic semigroups $\{P^t\}_{t\ge 0}$
which satisfy the following condition:

\noindent (K) for every $x_0\in X$ there exist  an $\varepsilon >0$,  a $t>0$,
and a measurable function 
$\eta\ge 0$ such that $\int \eta(x)\, m(dx)>0$ and
\begin{equation}
\label{w-eta3}
k(t,x,y)\ge \eta(y) \quad \textrm{for $x\in B(x_0,\varepsilon)$, $y\in X$},
\end{equation}
where
$B(x_0,\varepsilon)=\{x\in X:\,\,\rho(x,x_0)<\varepsilon\}$.
It is clear that stochastic semigroups satisfying condition (K) are partially integral. 
Now we present the main theorem of \cite{PR-JMMA2016}. 

\begin{theorem}
 \label{th:2}
Let $\{P^t\}_{t\ge 0}$ be a continuous time stochastic semigroup which satisfies {\rm (K)}.
Then there exist 
a countable (possible empty)
set $J$, a family of invariant densities
$\{f^*_j\}_{j\in J}$ with disjoint supports $\{A_j\}_{j\in J}$,  and a family
$\{\alpha_j\}_{j\in J}$
of positive linear functionals  defined on $L^1$
 such that
\begin{itemize}
 \item[(i)] for every $j\in J$ and for every  $f\in L^1$ we have
\begin{equation}
\label{pomoc:ogolny}
 \lim_{t\to \infty}\|\mathbf 1_{A_j} P^tf-\alpha_j(f)f^*_j\|=0,
\end{equation}
\item[(ii)] if $Y=X\setminus \bigcup\limits _{j\in J} A_j$, then for every  $f\in L^1$ and for every compact set $F$ we have
\begin{equation}
\label{th2:sw}
\lim_{t\to \infty} \int\limits_{F\cap Y} P^tf(x)\, m(dx)=0.
 \end{equation}
 \end{itemize}
\end{theorem}
We note that not only the sets $A_j$, $j\in J$, are disjoint but also their closures are disjoint, which  
means that the measures $\mu_j(dx):=f^*_j(x)\,m(dx)$ have disjoint topological supports \cite{PR-SD}.

In order to formulate a corollary of Theorem \ref{th:2}, we need an auxiliary notion.
A stochastic semigroup $\{P^t\}_{t\ge 0}$ is
called \textit{sweeping}
from a set $B\in\Sigma$ if for every
 $f\in D$
\begin{equation*}
\lim_{t\to\infty}\int_B P^tf(x)\,m(dx)=0.
\end{equation*}

\begin{corollary}
\label{col-sw}
Assume that a stochastic semigroup $\{P^t\}_{t\ge 0}$ satisfies
condition {\rm (K)} and has no invariant densities.
Then $\{P^t\}_{t\ge 0}$ is sweeping from compact sets.
\end{corollary}
This result also holds for discrete time stochastic semigroups.

Theorem~\ref{th:2} allows us also to find conditions guaranteeing that a stochastic semigroup $\{P^t\}_{t\ge 0}$ satisfies the {\it Foguel alternative}, i.e. it
is asymptotically stable or sweeping from all compact sets.
An advantage of the formulation of some results in the form of the Foguel
alternative is that in order to show asymptotic stability we do not need to prove the existence of an invariant density. It is enough to check that the semigroup is not sweeping from compact sets then, automatically, the semigroup $\{P^t\}_{t\ge 0}$ is asymptotically stable.
A more comprehensive introduction to the subject of 
 convergence and asymptotic behaviour of
semigroups of operators can be found in~\cite{BobrowskiRudnicki}.

\subsection{Applications to piecewise deterministic Markov processes}
\label{ss:PDMPs}
Theorems from Section~\ref{ss:asym-decomp} are very useful in the study of stochastic semigroups generated by Markov processes. In the case of diffusion processes, the semigroups describing the evolution of the density of distributions of the process are integral and have continuous and positive kernels. Hence, Theorem~\ref{th:2} can be applied without additional conditions and we immediately obtain the Foguel alternative. 
Theorem~\ref{th:2} is satisfied by a semigroup induced by any Markov chains because the space is discrete. This semigroup satisfies the Foguel alternative if the chain is irreducible. It turns out that the theorem can also be applied to piecewise deterministic Markov processes (PDMPs).  We will now give a definition of a PDMP and explain why such processes lead to partially integral semigroups.

According to a non-rigorous definition by Davis \cite{davis84}, the class of PDMPs is a
general family of stochastic models covering virtually all non-diffusion applications. A more formal definition of a PDMP is the following.
A continuous time (homogeneous) Markov process $X(t)$
is a \textit{PDMP} if  there is an increasing sequence of random times $(t_n)$, 
called \textit{jumps}, such that  
the sample paths  of $X(t)$ 
are defined in a deterministic way in each interval $(t_n,t_{n+1})$.

We consider three types of jumps:
\begin{itemize}
\item[a)] the process  can jump to a new point,  
\item[b)] at the moment of jump it changes the dynamics which defines 
its trajectories, 
\item[c)] changes the dynamics when it hits some surface (a stochastic billiard).
\end{itemize}
 
The induced semigroup by a PDMP is generally not integral, but because the jumps are at random moments there is a ``blurring" of distributions. For example, from a distribution concentrated at a point we get an absolutely continuous distribution with respect to the Lebesgue measure, which leads to a partially integral  semigroup.

\begin{figure}[h]
\begin{center}
\includegraphics[height=5cm]{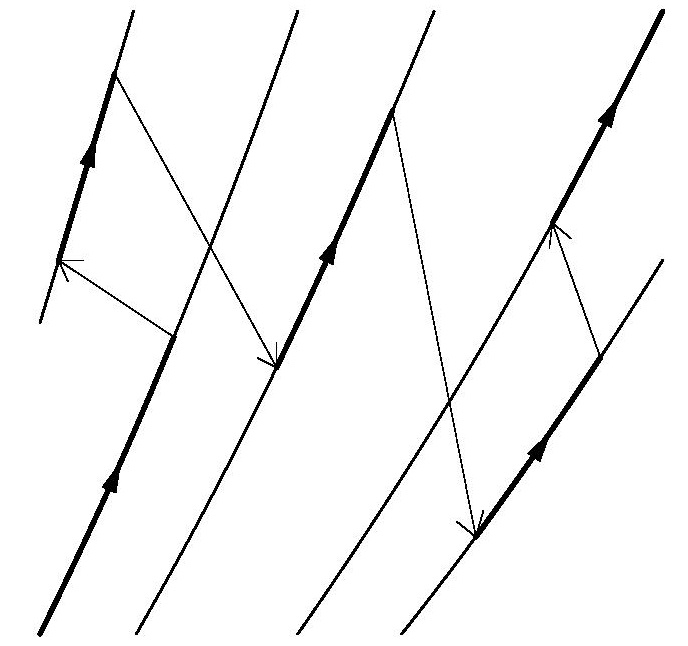}
\end{center}
\vskip-2mm
\caption{A flow with random jumps}
\label{r:zjazd2}
\end{figure}
\textit{Flows with jumps} (see Fig.~\ref{r:zjazd2}) are used in models of
population growth with disasters, immune systems \cite{PR-imun1,PR-imun2} and  cell cycle models \cite{PR-cell-cycle,PR-cell-2022}.

\textit{Processes with  switching dynamics} (see Fig.~\ref{r:zjazd3})
describe the action of several flows  with random switching between them.
They  are used as models of gene regulatory networks~\cite{BLPR,RT}. Depending
on the state of activity of the genes occurring in the network, such a model is described by
different systems of differential equations. A change in gene activity switches the system. 
\begin{figure}[h]
\begin{center}
\includegraphics[height=5cm]{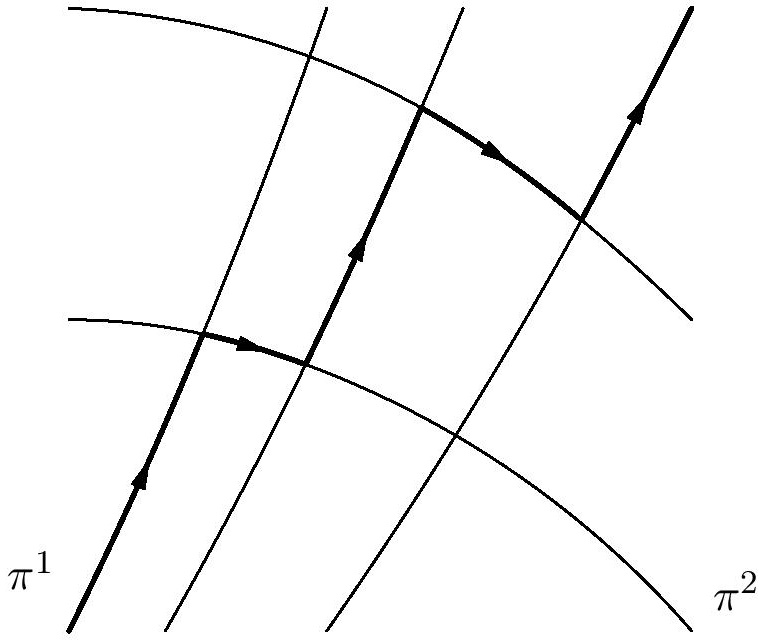}
\end{center}
\vskip-2mm
\caption{Process with  switching dynamics}
\label{r:zjazd3}
\end{figure}

We also consider 
processes in which a change in dynamics or a jump in phase space occurs when the process reaches the boundary of the domain or some fixed subset of the phase space. An example of such a process is a stochastic billiard~\cite{CPSV,LM-KR,M-KR}.
Similar properties have  models of cell cycle and  neuron activity~\cite{PR-Stein}.
A more comprehensive introduction to PDMPs and there applications 
can be found in~\cite{BiolMat2,RT-K-k}.

\section{Invariant measures for partial differential equations}
\label{s:im-pde}
\subsection{Lasota example}
\label{ss:L-ex}
It is hardly known that solutions of  simple 
linear partial differential equations 
behave in a chaotic way or have interesting ergodic properties.
Consider the following equation with the initial condition:   
\begin{equation}
\label{eq:lin}
\frac{\partial u}{\partial t}+
x\frac{\partial u}{\partial x}=\lambda u,\,\,\, 
u(0,x)=v(x), \,\,\,\lambda>0.
\end{equation}
This equation can served as a very simple model of the dynamics of cell maturation.
Here $x$ is the cell maturity, $1$  is the maturation rate. If $k$ is the rate of population growth, then  $\lambda=k-1$. 

Problem~\eqref{eq:lin} 
defines a semiflow on the space 
\[
X=\{v\in C[0,1]\colon v(0)=0\}
\]
given by 
\[
S^tv(x)=u(t,x)=
e^{\lambda t}v(e^{-t}x).
\] 

If the space is finite dimensional, then ergodic properties of semiflows can be successfully investigated by means of Frobenius--Perron operators. But this method is rather difficult
to apply in the infinite dimensional space~$X$. 
A. Lasota~\cite{Lasota-turbulencja} proposed to apply Theorem~\ref{miaraKB-w} to this flow
 and he proved 
that if $\lambda \ge 2$, then  there is a continuous ergodic measure
$\mu$ on $X$ invariant with respect to $\{S^t\}$.

The disadvantage of this method is that the constructed measure has a relatively small support and it is not mixing, 
and thus we cannot use it to prove the chaotic properties discussed in Sec.~\ref{ss:stoch}.  
In particular we are not able to prove that the semiflow $(S,\mu)$ has a dense orbit and has the property of sensitive dependence on initial conditions.

Looking for a method to construct invariant measures that have the desired properties, I noticed that it would be more convenient to use a method based on  isomorphism of the flow with a translation semiflow and using measure introduced by some stationary process.
We will present this method in the next section, but now we only mention that an invariant mixing measure with $\supp \mu=X$  can be given by the formula $\mu(A)={\rm P}(\xi_x\in A)$,
where $\xi_x=w_{x^{2\lambda}}$ and $w_x$ is the standard Wiener process.
Since problem~\eqref{eq:lin} describes the evolution of distribution of 
maturity in a cellular population one can prefer to consider 
the semiflow $\{S^t\}_{t\ge 0}$ restricted to the set 
$X_+=\{v\in X\colon v\ge 0\}$. In this case 
the invariant measure on $X_+$ can be induced by the process 
$\xi_x=|w_{x^{2\lambda}}|$.

\subsection{General approach}
\label{ss:gen-app} 
Now we consider  semiflow generated  by the equation of 
the form
\begin{equation}
\label{pde}
\frac{\partial u}{\partial t}+c(x)\frac{\partial
u}{\partial x}= f(x,u).
\end{equation}
We assume that $c$ and $f$ are $C^1$-functions and  
$c(0)=0$, $c(x)>0$ for $x\in (0,1]$,  
$f(0,u_*)=0$, $\,\frac{\partial f}{\partial u}(0,u_*)>0$, 
$\,f(0,u)\ne 0$ for $u\ne u_*$. 
Let 
\[
X=\{v\in C[0,1]: \ \ v(0)=u_* \}\quad\text{and}\quad
S_tv(x)=u(t,x).  
\]
This flow was studied  in~\cite{Rudnicki-pde1,Rudnicki-pde2} and, 
among other things, it was proved the following theorem.
\begin{theorem}
\label{miaraRRg} 
There exists a probability measure  $\mu$ 
such that: 

\noindent{\rm{(a)}} $\mu$  is invariant with respect to  $\{S^t\}$,\\
\noindent{{\rm(b)}} $\mu$ is exact,\\  
\noindent{{\rm(c)}} ${\rm supp\,}\mu=X$.
\end{theorem} 
\begin{proof}[Draft of the proof]
Let  $Y=C[0,\infty)$ and 
let $\{T^t\}_{t\ge 0}$ be the left-side shift on the space 
$Y$ defined by $(T^t \varphi)(s)=\varphi(s+t)$.\\
1. The  semiflows  
$\{S^t\}_{t\ge 0}$ and 
$\{T^t\}_{t\ge 0}$ 
 are isomorphic, i.e.
the map $Q: X \to Y$ given by $Qv(t)=S^tv(1)$ is 
a homeomorphism of $X$ onto 
$Y$ and 
\begin{equation}
\label{izomorfizm}
Q \circ S^t=T^t\circ Q \quad \text{for $\,t\ge 0$.}  
\end{equation}
2. Let $w_t$, $t\ge 0$,  be the Wiener  process 
and $\xi_t=e^t w_{e^{-2t}}$
for $t\ge 0$. Then $\xi_t$ is a stationary Gaussian
process with continuous trajectories. Let 
\[
m(A)=
{\rm P}\{
\omega: \xi_t(\omega)\in A
\},
\quad A\in{\mathcal B}(Y). 
\]
The measure $m$ is invariant under $\{T^t\}$
and $m(Q(X))=1$.  
Thus the measure $\mu(A)=m(Q(A))$ is invariant under $\{S^t\}$.\\
3. Exactness of $(T^t,Y)$ follows from 
the Blumenthal's zero-one law for the Wiener process;  from 
the definition of the measure $m$; and from an equivalent definition of exactness.\\
4. Positivity of  $\mu$ on open sets can be obtained from
the following property of Wiener process: 
\[
{\rm P}\{\omega: f(t)<w_t(\omega)<g(t)\ {\rm for\ } t\in[a,b]\} >0
\]
for continuous functions $f<g$ and $0<a<b$.
\end{proof}

\subsection{Other chaotic models}
\label{ss:other}
The method of constructing invariant measures presented in Section~\ref{ss:gen-app}  can be used for semiflows generated by other partial differential equations, although the proofs are no longer as easy as for Eq.~\eqref{pde}.
We will briefly present other equations studied using this method.

We consider a population of stem cells. 
These cells live in the bone marrow and they 
are precursors of any blood cells. 
They are subjects of two biological processes: maturation and division. 
Stem cells can be at different levels  of morphological development 
called maturity.  The maturity of a cell
differs from its age, because we assume that a newly born cell is 
in the same morphological state as its mother at the point of division. 
We assume that maturity is a real number $x\in [0,1]$. 
The function $u(t,x)$ describes the density distribution function 
of cells with respect to their maturity.
The maturity grows according to the  equation
$x'=g(x)$. When one cell reaches the maturity $1$ 
it leaves the bone marrow, then 
one of cells from the bone marrow splits. 
This cell is chosen randomly according to the distribution given by the density
$u(t,x)$. 
It follows from the assumptions that a newly born cell has the same maturity as its mother cell 
and each cell can divide with the same probability (see Fig.~\ref{r:model-chaos}).    
\begin{figure}[h]
\begin{center}
\begin{picture}(225,80)(20,-30)
\includegraphics[viewport=172 613 377 737]{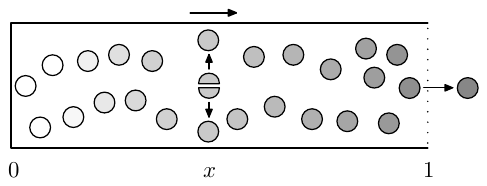} 
\end{picture}
\end{center}
\caption{Scheme of maturation and division of cells in the bone marrow.}
\label{r:model-chaos}
\end{figure}
There is an exterior regulatory system 
in which the production of erythrocytes is stimulated 
by the hormone erythropoietin and
the system tries to keep the number of erythrocytes on a constant level. 
We have not added this external regulatory system, 
thus we consider a pathological case when
the  exterior regulatory system does not work.
The model is described by a nonlinear semiflow
induced by the problem
\begin{equation}
\label{e1}
\frac{\partial u}{\partial t}+
\frac{\partial }{\partial x} (g(x)u)=
g(1)u(t,1)u(t,x), \quad u(0,x)=u_0(x).
\end{equation}
The semiflow is defined on the space of densities.
In the paper \cite{Rudnicki-chaos-den} it was shown 
that the semiflow generated by the initial problem 
(\ref{e1}) posses an invariant measure 
which is mixing and supported on the whole set of all densities.

The second example is the Bell and Anderson model of 
 size structured cellular  population given by the equation
\begin{equation}
\label{e22}
\frac{\partial u}{\partial t}+
\frac{\partial }{\partial x} (g(x)u)=
-(\mu+b)u(t,x)+4bu(t,2x),
\end{equation}
where $x\in [0,1]$ and we put $u(t,2x)=0$ if $2x>1$.
In \cite{Rudnicki-pd} it was shown 
that if $g(x)=ax$, then there exists a mixing invariant measure supported on the whole space.

The next example is  $d$-dimensional model considered 
in the paper~\cite{chaos2023}:
\[
\frac{\partial u}{\partial t}+a_1(x)\frac{\partial u}{\partial x_1}+\dots+a_d(x)\frac{\partial u}{\partial x_d}=f(x,u).
\]
In construction of invariant measure we use
\textit{L\'evy $d$-parameter Brownian motion}, which is a Gaussian random field, and we show that the semiflow is isomorphic with a translation flow along the radii.

The last example is the heat equation:
\begin{equation}
\label{E-heat1}
u_t(t,x) = u_{xx}(t,x),\quad t \ge 0,\quad x \ge 0
\end{equation}
with the boundary and initial conditions
\begin{equation}
\label{E-heat2}
u_x(t,0)=0\quad\text{for $t\ge 0$},\quad  u(0,x)=v(x)\quad\text{for $x\ge 0$}.    
\end{equation}
Here we construct an invariant mixing measure supported on the phase space
$X=\{v\in C[0,\infty)\colon \lim_{x\to\infty}e^{-x}v(x)=0\}$.

\section{Remarks and Conclusions}
\label{s:R+C}

\subsection{Other mathematical research inspired by biology}
\label{ss:other-res}
I devoted the lecture mainly to presenting some mathematical issues without going too far into discussing their applications to biological models. Finally, I would like to show with the example of A. Lasota and my experience, that the study of biological issues can be inspiring for the development of mathematical methods.
\vskip1mm
1. Cell cycle modeling has influenced the 
development of the theory of Markov operators and the study of their asymptotic properties.
\vskip1mm

2. Structural population models lead to the development of the theory of transport equations, other than those found in physics, and the theory of semigroup operators.
\vskip1mm

3. The study of population models involves non-trivial nonlinear and non-local partial differential equations with delays~\cite{MR-nl}.
Studying the behavior of solutions to such equations is quite a challenge for mathematicians. 

\vskip1mm

4. Biologists are increasingly using individual-based models, also known as agent-based models. Individual-based models describe a population as a collection of different organisms
whose local interactions determine the behaviour of the entire population. The individual description is convenient for computer simulations and the determination
of various model parameters, and appropriate limit passages lead to interesting transport equations.
For example, the phenotypic model studied in the paper~\cite{RZ} after an appropriate limit passage leads to a transport equation,   including the Tjon-Wu equation of the energy distribution of particles in the Boltzmann equation,
which was studied by A. Lasota~\cite{lasota-Tjon--Wu}.

\subsection{To take home}\ \ 
\label{ss:ohome}
\vskip1mm

1. Markov operators and semigroups can be used to describe and prove ergodic and chaotic properties of dynamical systems.

\vskip1mm

2. Markov semigroups describe the evolutions of densities of Markov processes: diffusion processes, piecewise deterministic Markov processes, stochastic hybrid systems (diffusions with jumps).

\vskip1mm

3. The chaotic properties of dynamical systems generated by partial differential equations can be successfully studied using stochastic processes.

\vskip1mm

4. Modern biology boldly reaches for mathematical models, and virtually every part of mathematics can be useful in~modeling. Due to the random nature of biological phenomena, 
stochastics plays a major role in this regard.  
\vskip1mm

5. We will conclude with a quote from Sir Rory Collins head of clinical research at Oxford University:  ``\textbf{If you want a career in medicine these days you're better off studying mathematics or computing than biology.}"

\end{document}